\documentclass[notitlepage,11pt,reqno]{amsart}
\usepackage[foot]{amsaddr}
\usepackage[numbers,sort]{natbib}

\usepackage{amssymb,nicefrac,bm,upgreek,mathtools,verbatim}
\usepackage[final]{hyperref}
\usepackage[mathscr]{eucal}
\usepackage{dsfont}

\usepackage{color}
\definecolor{dmagenta}{rgb}{.4,.1,.5}
\definecolor{dblue}{rgb}{.0,.0,.5}
\definecolor{mblue}{rgb}{.0,.0,.8}
\definecolor{ddblue}{rgb}{.0,.0,.4}
\definecolor{dred}{rgb}{.6,.0,.0}
\definecolor{dgreen}{rgb}{.0,.5,.0}
\definecolor{Eeom}{rgb}{.0,.0,.5}

\usepackage[normalem]{ulem}
\usepackage[margin=1in]{geometry}
\allowdisplaybreaks
\newtheorem{lemma}{Lemma}[section]
\newtheorem{theorem}{Theorem}[section]

\newtheorem{corollary}{Corollary}[section]

\theoremstyle{definition}
\newtheorem{definition}{Definition}[section]
\newtheorem{assumption}{Assumption}[section]

\newtheorem{example}{Example}[section]
\newtheorem{remark}{Remark}[section]

\numberwithin{equation}{section}

\hypersetup{
  colorlinks=true,
  citecolor=dblue,
  linkcolor=dblue,
  frenchlinks=false,
  pdfborder={0 0 0},
  naturalnames=false,
  hypertexnames=false,
  breaklinks}
\usepackage[capitalize,nameinlink]{cleveref}

\crefname{section}{Section}{Sections}
\crefname{subsection}{Subsection}{Subsections}
\crefname{condition}{Condition}{Conditions}
\crefname{hypothesis}{Hypothesis}{Conditions}
\crefname{assumption}{Assumption}{Assumptions}
\crefname{lemma}{Lemma}{Lemmas}
\crefname{claim}{Claim}{Claims}

\Crefname{figure}{Figure}{Figures}

\crefformat{equation}{\textup{#2(#1)#3}}
\crefrangeformat{equation}{\textup{#3(#1)#4--#5(#2)#6}}
\crefmultiformat{equation}{\textup{#2(#1)#3}}{ and \textup{#2(#1)#3}}
{, \textup{#2(#1)#3}}{, and \textup{#2(#1)#3}}
\crefrangemultiformat{equation}{\textup{#3(#1)#4--#5(#2)#6}}%
{ and \textup{#3(#1)#4--#5(#2)#6}}{, \textup{#3(#1)#4--#5(#2)#6}}%
{, and \textup{#3(#1)#4--#5(#2)#6}}

\Crefformat{equation}{#2Equation~\textup{(#1)}#3}
\Crefrangeformat{equation}{Equations~\textup{#3(#1)#4--#5(#2)#6}}
\Crefmultiformat{equation}{Equations~\textup{#2(#1)#3}}{ and \textup{#2(#1)#3}}
{, \textup{#2(#1)#3}}{, and \textup{#2(#1)#3}}
\Crefrangemultiformat{equation}{Equations~\textup{#3(#1)#4--#5(#2)#6}}%
{ and \textup{#3(#1)#4--#5(#2)#6}}{, \textup{#3(#1)#4--#5(#2)#6}}%
{, and \textup{#3(#1)#4--#5(#2)#6}}

\crefdefaultlabelformat{#2\textup{#1}#3}

\makeatletter
\DeclareRobustCommand\widecheck[1]{{\mathpalette\@widecheck{#1}}}
\def\@widecheck#1#2{%
    \setbox\z@\hbox{\m@th$#1#2$}%
    \setbox\tw@\hbox{\m@th$#1%
       \widehat{%
          \vrule\@width\z@\@height\ht\z@
          \vrule\@height\z@\@width\wd\z@}$}%
    \dp\tw@-\ht\z@
    \@tempdima\ht\z@ \advance\@tempdima2\ht\tw@ \divide\@tempdima\thr@@
    \setbox\tw@\hbox{%
       \raise\@tempdima\hbox{\scalebox{1}[-1]{\lower\@tempdima\box
\tw@}}}%
    {\ooalign{\box\tw@ \cr \box\z@}}}

\def\subsection{\@startsection{subsection}{0}%
\z@{\linespacing\@plus\linespacing}{\linespacing}%
{\bf}}

\makeatother

\DeclareMathOperator{\Exp}{\mathbb{E}} %Expectation
\DeclareMathOperator{\Prob}{\mathbb{P}} %Probability
\newcommand{\D}{\mathrm{d}}          %differential
\newcommand{\RR}{\mathbb{R}}         %Real numbers
\newcommand{\Rd}{{\mathbb{R}^d}}       %R^d
\newcommand{\Ind}{\mathds{1}}            %indicator function
\newcommand{\grad}{\nabla}

\newcommand{\sB}{\mathscr{B}}    % Ball used often
\newcommand{\sU}{\mathscr{U}}
\newcommand{\cC}{\mathcal{C}}     % Classes of continuous functions
\newcommand{\cD}{\mathcal{D}}     % Domain of \Rd
\newcommand{\sH}{\mathscr{H}}

\newcommand{\cK}{\mathcal{K}}

\newcommand{\sR}{\mathscr{R}}
\newcommand{\abs}[1]{\lvert#1\rvert}
\newcommand{\norm}[1]{\lVert#1\rVert}

\providecommand{\pro}[1]{(#1_t)_{t \geq 0}}

\providecommand{\semi}[1]{\{#1_t: t \geq 0\}}
\providecommand{\seq}[1]{(#1_n)_{n\in \mathbb{N}}}

\DeclareMathOperator{\Tr}{Tr}
\DeclareMathOperator{\diam}{diam}

\newcommand{\ex}{\mathbb{E}}

\DeclareMathOperator{\dist}{dist}

\DeclareMathOperator{\Psidel}{\Psi(--\Delta)}

\begin{document}

\title[Maximum principles]%
{\sc \textbf{Hopf's lemma for viscosity solutions to a class of non-local equations with applications}}

\author{Anup Biswas and J\'{o}zsef L\H{o}rinczi}

\address{Anup Biswas \\
Department of Mathematics, Indian Institute of Science Education and Research, Dr. Homi Bhabha Road,
Pune 411008, India, anup@iiserpune.ac.in}

\address{J\'ozsef L\H{o}rinczi \\
Department of Mathematical Sciences, Loughborough University, Loughborough LE11 3TU, United Kingdom,
J.Lorinczi@lboro.ac.uk}

\date{}

%%%%%%%%%%%%%%%%%%%%%%%%%%%%%%%%%%%%%%%%%%%%%%%%%%%%%%%%%%%%%%%%%%%%%%%%%%%%%%%%

\begin{abstract}
We consider a large family of integro-differential equations
%featuring Markov generators of L\'evy processes,
and establish a non-local counterpart of Hopf's lemma, directly expressed in terms of the symbol of the operator.
As closely related problems, we also obtain a variety of maximum principles for viscosity solutions. In our
approach we combine direct analysis with functional integration, allowing a robust control around the boundary
of the domain, and make use of the related ascending ladder height-processes. We then apply these results to a
study of principal eigenvalue problems, the radial symmetry of the positive solutions, and the overdetermined
non-local torsion equation.
\end{abstract}

\maketitle

\section{\bf Introduction}
Hopf's boundary point lemma is a classic result in analysis, belonging to the range of maximum principles,
and it proved to be a fundamental and powerful tool in the study of partial differential equations. For a
general introduction we refer to \cite{PW}, and to the magisterial paper \cite{PS} for a more modern reassessment
and further developments.
It is a natural question whether a variant of Hopf's lemma with a similar benefit might be possible to obtain
for integro-differential equations. Such non-local equations and related problems are currently much researched
in both pure and applied mathematics, also attracting a wide range of applications in the natural sciences and
elsewhere.

Our aim in the present paper is to derive and prove Hopf's lemma and related maximum principles for a class of
non-local equations in which the key operator term is a Bernstein function of the Laplacian, denoted below by
$\Psi(-\Delta)$. (For precise definitions see Section 2.) This class is increasingly used in various
directions such as generalizations of the Caffarelli-Silvestre extension technique \cite{KM}, the blow-up of
solutions of stochastic PDE with white or coloured noise \cite{DLN}, scattering theory and spectral thresholds
\cite{IW20}, or the potential theory of subordinate Brownian motions \cite{SSV,KSV}.

There are good reasons why this more general framework is of interest. One is that there is a large class of
operators with singular jump kernels whose elements are comparable with some $\Psi(-\Delta)$ in a specific sense,
see for details \cite[Th. 26, Cor. 27]{BGR14b}. Another is that diverse choices of Bernstein functions often
produce in various aspects qualitatively different behaviours. The fractional Laplacian obtained for $\Psi(u) =
u^{\alpha/2}$, $0 < \alpha < 2$, appears to be the most often featured in the literature, and is a fundamental
example of a non-local operator with a heavy-tailed jump measure. For many applications, however, operators with
exponentially light jump kernels are also of similar interest, for which a basic example is the relativistic operator
with exponent $\frac{\alpha}{2}$ and rest mass parameter $m>0$, obtained for $\Psi(u) = (u + m^{2/\alpha})^{\alpha/2}
-m$, see a summary and applications in \cite{LS}. For some more choices of particular interest we refer to Example
\ref{Eg2.1} below.
One can see these different behaviours on the case of non-local
Schr\"odinger operators of the form $H=\Psi(-\Delta)+V$, with a multiplication operator $V$, and observe how the
relativistic operator plays in a particular sense an intermediary role between the fractional and the classical
Laplacian. Some new phenomena, which do not occur for Schr\"odinger operators with the usual Laplacian, are ground
state domination on all the eigenfunctions \cite{KL15}, a sharp ``phase transition" in the decay rates of the
eigenfunctions as the L\'evy measure is progressively changed from heavy to light tails \cite{KL17}, eigenfunction
decay determined by global lifetimes in cases of zero eigenvalues for  potentials $V$ decaying to zero at infinity
instead of local lifetimes for operators generating negative eigenvalues \cite{KL19}, moment excess-driven phenomena
for exponentially light-tailed operators \cite{AL}, or varying contractivity properties of the evolution semigroups
$\{e^{-tH}: t \geq 0\}$ \cite{KKL}.

A closely related interest is a study of non-linear fractional Schr\"odinger operators in bounded domains or full
space; for a general survey of exterior value problems we refer to \cite{RO} and the references therein.
Non-linearities introduce new aspects such as criticality, layer and radial solutions \cite{CS14,CS15}, concentration
phenomena \cite{FQT,S13,DDW,A15,A19}, Ambrosetti-Prodi type bifurcations \cite{DQT,GS16,BL20}, structural instability
due to non-additive energies \cite{DKV,DV}, to name a few.

In the study of exterior value problems the non-local counterpart of the classical Hopf's lemma is of fundamental
interest. A first result in this direction has been obtained in \cite{FJ15, GS16}, where the authors proved it for
Dirichlet exterior value problems involving the fractional Laplacian $(-\Delta)^{\alpha/2}$, $\alpha \in
(0,2)$. The problem has been studied for the fractional $p$-Laplacian in \cite{CLQ, DQ17}.

The purpose of this paper is to obtain Hopf's lemma in the generality of non-local operators given by Bernstein
functions of the Laplacian $\Psi(-\Delta)$ and various choices of the non-linear source term, so that we have
expressions directly involving the symbol of the operator. Our approach here combines analytic and stochastic
tools based on a probabilistic representation of the operator semigroups, which proved to be robust in tackling
important basic difficulties encountered in the control close to the boundary of the domain when purely analytic
techniques are applied \cite{BL17a,BL17b,BL20}. A further highlight of our approach is a use of the ascending
ladder height-process related to the random process generated by $-\Psi(-\Delta)$, which has not been much
explored in the literature in this context. Apart from covering a large class of equations, another main technical
step forward is that our results are valid for viscosity solutions, while even for the fractional Laplacian the
results in \cite{GS16} have been established for classical solutions, and in \cite{DQ17, FJ15} for continuous,
non-negative weak super-solutions only.

Let $\cD\subset\Rd$, $d\geq 2,$ be a bounded domain with a $\cC^{1,1}$ boundary, and $c, f$ given continuous
functions. We will be interested in the viscosity solutions of
\begin{equation}\label{E2.6}
\left\{\begin{array}{ll}
-\Psidel u(x) + c(x) u(x) = f(x)  \quad \text{in} \; \cD,
\\[2mm]
\hspace{3.9cm} u = 0  \qquad\;\, \text{in}\; \cD^c.
\end{array}
\right.
\end{equation}
%\begin{equation}\label{E2.6}
%\begin{split}
%-\Psidel u(x) + c(x) u(x) &= f(x) \quad \text{in} \; \cD,
%\\
%u &= 0 \quad \text{in}\; \cD^c.
%\end{split}
%\end{equation}
First in Theorem \ref{T3.1} we obtain an Aleksandrov-Bakelman-Pucci type estimate for viscosity solutions for the
above non-local problem, which immediately implies a maximum principle for narrow domains shown in \cref{C3.1}.
A second result, presented in \cref{T2.1}, establishes existence and uniqueness for the principal eigenfunction
for the operator $-\Psidel + c(x)$, again in viscosity sense. Next we prove Hopf's lemma in \cref{T2.2} and
identify the function of the distance to the boundary, directly expressed in terms of $\Psi$, replacing the
normal derivative in the classical variant of the result. Then we turn to proving a refined maximum principle in
\cref{T2.3}, an anti-maximum principle in \cref{T2.4}, and in \cref{T3.3} we obtain that the principal eigenvalue
of the non-local Schr\"odinger operator $-\Psidel + c$ with Dirichlet exterior condition is an isolated eigenvalue.

In the final section of this paper we also present two applications of these maximum principles. One direction is
to show radial symmetry of positive viscosity solutions  in symmetric domains. This will be discussed
in \cref{T3.4}. A second application is to the torsion function, which is a quantity of interest in mathematical
physics, and also plays a significant role in probability, corresponding to mean exit times from domains. In
\cref{T4.2} we discuss the torsion equation in our context, over-determined by a constraint imposed on the domain
boundary, which is a non-local development of a classical inverse problem by Serrin and Weinberger \cite{S71, W71}. 
Some recent works considering overdetermined problems for nonlocal operators include \cite{FJ15,GS16,SV19} for the
fractional Laplacian, and \cite{BJ20} for the operators $\Psidel$ considered in the present paper.

\section{\bf Bernstein functions of the Laplacian and subordinate Brownian motion}
\subsection{Bernstein functions and subordinate Brownian motion}
In this section we briefly recall the essentials of the framework we use in this paper.
A Bernstein function is a non-negative completely monotone function, i.e., an element of the set
$$
\mathcal B = \left\{f \in C^\infty((0,\infty)): \, f \geq 0 \;\; \mbox{and} \:\; (-1)^n\frac{d^n f}{dx^n} \leq 0,
\; \mbox{for all $n \in \mathbb N$}\right\}.
$$
In particular, Bernstein functions are increasing and concave. We will make use below of the subset
$$
{\mathcal B}_0 = \left\{f \in \mathcal B: \, \lim_{x\downarrow 0} f(x) = 0 \right\}.
$$
Let $\mathcal M$ be the set of Borel measures $\mathfrak m$ on $\RR \setminus \{0\}$ with the property that
$$
\mathfrak m ((-\infty,0)) = 0 \quad \mbox{and} \quad \int_{\RR\setminus\{0\}} (y \wedge 1)  \, \mathfrak m(dy) < \infty.
$$
Notice that, in particular, $\int_{\RR\setminus\{0\}} (y^2 \wedge 1) \, \mathfrak m(dy) < \infty$ holds, thus
$\mathfrak m$ is a L\'evy measure supported on the positive semi-axis. It is well-known that every Bernstein function
$\Psi \in {\mathcal B}_0$ can be represented in the form
\begin{equation*}
\Psi(x) = bx + \int_{(0,\infty)} (1 - e^{-yx}) \, \mathfrak m(\D{y})
\end{equation*}
with $b \geq 0$, moreover, the map $[0,\infty) \times \mathcal M \ni (b,\mathfrak m) \mapsto \Psi \in
{\mathcal B}_0$ is bijective \cite[Th.~3.2]{SSV}. Also, $\Psi$ is said to be a complete Bernstein function if there 
exists a Bernstein function $\widetilde\Psi$ such that
$$
\Psi(x)= x^2 \mathcal{L}(\widetilde\Psi)(x), \quad x>0\,,
$$
where $\mathcal{L}$ denotes Laplace transform. Every complete Bernstein function is also a Bernstein function, and the
L\'evy measure $\mathfrak m$ of a complete Bernstein function has a completely monotone density with respect to
Lebesgue measure \cite[Th.~6.2]{SSV}. For a detailed discussion of Bernstein functions we refer to the monograph \cite{SSV}.

Bernstein functions are closely related to subordinators. Recall that a one-dimensional L\'evy process $\pro S$ on a
probability space $(\Omega_S, {\mathcal F}_S, \mathbb P_S)$ is called a subordinator whenever it satisfies $S_s \leq S_t$
for $s \leq t$, $\mathbb P_S$-almost surely. A basic fact is that the Laplace transform of a subordinator is given by a
Bernstein function, i.e.,
\begin{equation*}
\label{lapla}
\ex_{\mathbb P_S} [e^{-xS_t}] = e^{-t\Psi(x)}, \quad t, x \geq 0,
\end{equation*}
holds, where $\Psi \in {\mathcal B}_0$. In particular, there is a bijection between the set of subordinators on a given
probability space and Bernstein functions with vanishing right limits at zero.
% to emphasize this, we will occasionally
%write $\pro {S^\Psi}$ for the unique subordinator associated with Bernstein function $\Psi$.

Let $\pro B$ be an $\Rd$-valued Brownian motion on Wiener space $(\Omega_W,{\mathcal F}_W, \mathbb P_W)$, running twice
as fast as standard $d$-dimensional Brownian motion, and let $\pro {S}$ be an independent subordinator with characteristic
exponent $\Psi$. The random process
$$
\Omega_W \times \Omega_S \ni (\omega_1,\omega_2) \mapsto B_{S_t(\omega_2)}(\omega_1) \in \Rd
$$
is called subordinate Brownian motion under $\pro {S}$. For simplicity, we will denote a subordinate Brownian motion
by $\pro X$, its probability measure for the process starting at $x \in \Rd$ by $\mathbb P^x$, and expectation with respect
to this measure by $\ex^x$. Note that the characteristic exponent of a pure jump process $\pro X$ (i.e., with $b=0$) is given
by
\begin{equation}\label{E2.4}
\Psi(\abs{z}^2)= \int_{\Rd\setminus\{0\}} (1-\cos(y\cdot z)) j(\abs{y}) \, \D{y},
\end{equation}
where the L\'evy measure of $\pro X$ has a density $y\mapsto j(\abs{y})$, $j:(0, \infty)\to (0, \infty)$, with respect to
Lebesgue measure, given by
\begin{equation}\label{E2.3}
j(r) = \int_0^\infty (4\pi t)^{-d/2} e^{-\frac{r^2}{4t}} \, \mathfrak m (\D{t}).
\end{equation}

Below we will use Bernstein functions satisfying the following conditions. These have been extensively used in \cite{BL17a,BL17b},
and for applications in potential theory we refer to \cite{BGR14b}.
\begin{definition}
{\rm
\label{WLSC}
The function $\Psi \in {\mathcal B}_0$ is said to satisfy a
\begin{enumerate}
\item[(i)]
\emph{weak lower scaling (WLSC) property} with parameters ${\underline\mu} > 0$, $\underline{c} \in(0, 1]$ and
$\underline{\theta}\geq 0$, if
$$
\Psi(\gamma x) \;\geq\; \underline{c}\, \gamma^{\underline\mu} \Psi(x), \quad x>\underline{\theta}, \; \gamma\geq 1.
$$
\item[(ii)]
\emph{weak upper scaling (WUSC) property} with parameters $\bar\mu > 0$, $\bar{c} \in[1, \infty)$ and $\bar{\theta}\geq 0$,
if
$$
\Psi(\gamma x) \;\leq\; \bar{c}\, \gamma^{\bar{\mu}} \Psi(x), \quad x>\bar{\theta}, \; \gamma\geq 1.
$$
\end{enumerate}
}
\end{definition}
\noindent
Clearly, we have $\bar{\mu}\geq \underline{\mu}$.

\begin{example}\label{Eg2.1}
{\rm
Some important examples of $\Psi$ satisfying WLSC and WUSC include the following
cases with the given parameters, respectively:
\begin{itemize}
\item[(i)]
$\Psi(x)=x^{\alpha/2}, \, \alpha\in(0, 2]$, with ${\underline\mu} = \frac{\alpha}{2}$, $\underline{\theta}=0$,
and $\bar\mu = \frac{\alpha}{2}$, $\bar{\theta}=0$.
\item[(ii)]
$\Psi(x)=(x+m^{2/\alpha})^{\alpha/2}-m$, $m> 0$, $\alpha\in (0, 2)$, with ${\underline\mu} = \frac{\alpha}{2}$,
$\underline{\theta}=0$, and $\bar\mu = 1$, $\bar{\theta}=0$ and $\bar\mu = \frac{\alpha}{2}$ for any $\bar{\theta}>0$.
\item[(iii)]
$\Psi(x)=x^{\alpha/2} + x^{\beta/2}, \, \alpha, \beta \in(0, 2]$, with ${\underline\mu} = \frac{\alpha}{2}
\wedge \frac{\beta}{2}$, $\underline{\theta}=0$ and $\bar\mu = \frac{\alpha}{2} \vee \frac{\beta}{2}$,
$\bar{\theta}=0$.
\item[(iv)]
$\Psi(x)=x^{\alpha/2}(\log(1+x))^{-\beta/2}$, $\alpha \in (0,2]$, $\beta \in [0,\alpha)$ with
${\underline\mu}=\frac{\alpha-\beta}{2}$, $\underline{\theta}=0$ and $\bar\mu=\frac{\alpha}{2}$, $\bar{\theta}=0$.
\item[(v)]
$\Psi(x)=x^{\alpha/2}(\log(1+x))^{\beta/2}$, $\alpha \in (0,2)$, $\beta \in (0, 2-\alpha)$, with
${\underline\mu}=\frac{\alpha}{2}$, $\underline{\theta}=0$ and $\bar\mu=\frac{\alpha+\beta}{2}$, $\bar{\theta}=0$.
\end{itemize}
The above are complete Bernstein functions, and an example of a Bernstein function which is not complete is
$1 - e^{-x}$, describing the Poisson subordinator. Corresponding to the examples above, the related processes are
(i) $\frac{\alpha}{2}$-stable subordinator, (ii) relativistic $\frac{\alpha}{2}$-stable subordinator, (iii) sums
of independent subordinators of different indices,
%(iv) geometric $\frac{\alpha}{2}$-stable subordinators (specifically, the Gamma-subordinator for $\alpha = 2$),
etc.
}
\end{example}
We will use below the following recurring assumptions.
\begin{assumption}\label{A2.1}
{\rm
$\Psi \in {\mathcal B}_0$ satisfies both the WLSC and WUSC properties with respect to suitable values
$(\underline{\mu}, \underline{c}, \underline{\theta})$ and $(\bar{\mu}, \bar{c}, \bar{\theta})$, respectively,
with $0<\underline{\mu}\leq\bar{\mu}<1$.
}
\end{assumption}

A second assumption is on the L\'evy jump kernel of the subordinate process.
\begin{assumption}\label{A2.2}
{\rm
There exists a constant $\varrho > 0$ such that
\begin{equation}\label{E2.5}
j(r+1)\geq \varrho\, j(r),\quad \text{for all}\; \; r\geq 1,
\end{equation}
where $j$ is given by \eqref{E2.3}.
}
\end{assumption}
\noindent
Note that if $\Psi$ is a complete Bernstein function and satisfies $\Psi(r)\asymp r^{\gamma}\ell(r)$ as $r\to\infty$,
for a suitable $\gamma\in(0, 1)$ and a locally bounded and slowly varying function $\ell$, then \eqref{E2.5} holds
\cite[Th. 13.3.5]{KSV}.

For some of our proofs below we will use some information on the normalized ascending ladder-height process of
$\pro {X^1}$, where $X^{1}_t$ denotes the first coordinate of $X_t$. Recall that the ascending ladder-height
process of a L\'evy process $\pro Z$ is the process of the right inverse $(Z_{L^{-1}_t})_{t\geq 0}$, where $L_t$
is the local time of $Z_t$ reflected at its supremum (for details and further information we refer to \cite[Ch. 6]{B}).
Also, we note that the ladder-height process of $\pro {X^1}$ is a subordinator with Laplace exponent
$$
\tilde\Psi(x)=\exp\left(\frac{1}{\pi}\int_0^\infty \frac{\log \Psi(x^2y^2)}{1+y^2}\, \D{y}\right), \quad x \geq 0.
$$
Consider the potential measure $V(x)$ of this process on the half-line $(-\infty, x)$. Its Laplace transform is
given by
$$
\int_0^\infty V(x) e^{-sx}\, \D{x}= \frac{1}{s\tilde\Psi(s)}, \quad s > 0.
$$
It is also known that $V=0$ for $x\leq 0$, the function $V$ is continuous and strictly increasing in $(0, \infty)$
and $V(\infty)=\infty$ (see \cite{F74} for more details). As shown in \cite[Lem.~1.2]{BGR14} and \cite[Cor.~3]{BGR14b},
there exists a constant $C = C(d)$ such that
\begin{equation}\label{E3.5}
\frac{1}{C}\,{\Psi(1/r^2)}\leq \frac{1}{V^2(r)}\leq C\, {\Psi(1/r^2)}, \quad r>0.
\end{equation}
Using \cite[Th.~4.6, Lem.~7.3]{BGR15} and \cref{A2.1}, we see that for a suitable positive constant $\kappa_1$ we have that for
$x\in\cD$
\begin{align}
\Exp^x[\uptau_\cD] &\geq \kappa_1 V(\delta_\cD(x))
\label{E3.6}
\end{align}
holds, where $\delta_\cD(x)=\dist (x, \partial \cD)$ and
$$
\uptau_{\cD} = \inf\{t > 0: \, X_t \not\in \cD\}
$$
denotes the first exit time of $\pro X$ from $\cD$.

\subsection{Bernstein functions of the Laplacian}
From now on we consider
$$
\Psi \in {\mathcal B}_0 \quad \mbox{with} \quad b=0.
$$
Using \eqref{E2.4}, we define the operator
\begin{align*}
-\Psidel f(x) &= \int_{\Rd} \left(f(x+z)-f(x)-\Ind_{\{\abs{z}\leq 1\}} z\cdot \grad f(x)\right) j(\abs{z}) dz
\\
&=\frac{1}{2} \int_{\Rd} (f(x+z)+f(x-z)-2 f(x)) j(\abs{z}) dz,
\end{align*}
for $f\in\cC^2_{\rm b}(\Rd)$, by functional calculus. We use the notations $\cC_{\rm b}(\Rd)$ for the space of 
bounded continuous functions on $\Rd$, and $\cC^2_{\rm b}(\Rd)$ for the space of twice continuously differentiable 
bounded continuous functions on $\Rd$.
The operator $-\Psidel$ is the Markov generator of
subordinate Brownian motion $\pro X$ corresponding to the subordinator uniquely determined by $\Psi$, i.e.,
$$
{e^{-t\Psidel}}f(x) = \ex^x [f(X_t)], \quad t \geq 0, \, x \in \Rd, \, f \in L^2(\Rd).
$$

Next consider a bounded domain $\cD \subset \Rd$, and the space $C_{\rm c}^\infty(\cD)$. We can define the linear
operator $-\Psidel^{\cD}$ on this domain, given by the Friedrichs extension of $-\Psidel|_{C_{\rm c}^\infty(\cD)}$.
It can be shown that the form-domain of $\Psidel^{\cD}$ contains all functions which are in the form-domain of
$-\Psidel$ and almost surely zero outside of $\cD$. To ease the notation, from now on we use the simple notation
$-\Psidel$ also on $\cD$, understanding it to be this operator.

Let now $c: \cD \to \RR$ be a bounded continuous function, and define it as a multiplication operator on $C_{\rm c}^\infty(\cD)$.
We consider the operator $-\Psidel + c$ on $L^2(\cD)$, again as the Friedrichs extension of the operator sum
$(-\Psidel + c)|_{C_{\rm c}^\infty(\cD)}$.  Define the operator
$$
T_t f(x) =\Exp^x\left[e^{\int_0^t c(X_s) ds} f(X_t) \Ind_{\{t<\uptau_{\cD}\}}\right], \quad t\geq 0\,.
$$
It is shown in \cite[Lem.~3.1]{BL17a} that $\{T_t\; :\; t\geq 0\}$
is a strongly continuous semigroup on $L^2(\cD)$, with infinitesimal generator $-\Psidel + c$. Probabilistically,
this means that $\semi {T}$ is the Markov semigroup of subordinate Brownian motion killed upon exit from $\cD$. 
Moreover, there exists a pair $(\psi^*, \lambda)\in\cC_{{\rm b}}(\Rd)\times\RR$, $\psi^*>0$ in $\cD$, satisfying
\begin{equation}\label{E3.2}
\psi^*(x) = \Exp^x\left[e^{\int_0^t (c(X_s)-\lambda) ds} \psi^*(X_t) \Ind_{\{t<\uptau_{\cD}\}}\right], \quad t>0, \;
x\in \cD,
\end{equation}
and $\psi^*(x)=0$ in $\cD^c$. For further details we refer the reader to \cite{BL17a}.

\section{\bf Hopf's lemma and maximum principles for non-local equations}\label{S-statement}
\subsection{Aleksandrov-Bakelman-Pucci estimate and Hopf's lemma for viscosity solutions}
Let $\cD\subset\Rd$, $d\geq 2,$ be a bounded domain with a $\cC^{1,1}$ boundary. With no loss of generality
we assume that $0\in \cD$. With given continuous functions $c$ and $f$, our purpose is to consider viscosity
solutions of the Dirichlet exterior value problem \eqref{E2.6}.
%\begin{split}
%-\Psidel u(x) + c(x) u(x) &= f(x) \quad \text{in} \; D,
%\\
%u &= 0 \quad \text{in}\; D^c.
%\end{split}
%\end{equation}

Recall the definition of a viscosity solution. Denote by $\cC^2_{\rm b}(x)$ the space of bounded functions, twice
continuously differentiable in a neighbourhood of $x \in \Rd$.

\begin{definition}
{\rm
An upper semi-continuous function $u:\Rd\to\RR$ in $\bar \cD$ is said to be a
\emph{viscosity sub-solution} of
\begin{equation}
\label{visco}
-\Psidel u + c\, u =  f\quad \text{in}\; \cD,
\end{equation}
if for every $x\in \cD$ and test function $\varphi\in \cC^2_{\rm b}(x)$ satisfying $u(x)=\varphi(x)$ and
$\varphi(y) > u(y)$, $y\in \Rd\setminus\{x\}$, we have
$$
-\Psidel \varphi(x) + c(x) u(x) \geq  f(x).
$$
Similarly, a lower semi-continuous function is a \emph{viscosity super-solution} of \eqref{visco} whenever
$\varphi(y) < u(y)$, $y\in \Rd\setminus\{x\}$, implies $-\Psidel \varphi(x) + c(x) u(x) \leq  f(x)$.
Furthermore, $u$ is said to be a \emph{viscosity solution} if it is both a viscosity sub- and super-solution.
}
\end{definition}

One of our main tools is an Aleksandrov-Bakelman-Pucci type maximum principle.
\begin{theorem}[\textbf{ABP-type estimate}]
\label{T3.1}
Suppose that $\Psi$ satisfies the WLSC property with parameters $(\underline{\mu}, \underline{c}, \underline{\theta})$.
Let $f: \cD \to \RR$ be a continuous function, and $u\in\cC_{\rm b}(\Rd)$ a viscosity sub-solution of
$$
-\Psidel u = - f\quad \text{in}\;\; \{u>0\}\cap \cD, \quad \text{and}\quad u\leq 0\quad \text{in}\; \; \cD^c.
$$
Then for every $p>\frac{d}{2\underline{\mu}}$ there exists a constant $C = C(p,\Psi,\diam \cD)$, such that
\begin{equation}\label{ET3.1A}
\sup_{\cD} u^+ \;\leq \; C \norm{f^+}_{L^p(\cD)}.
\end{equation}
\end{theorem}

\begin{proof}
Let $\cD_1=\{u>0\}\cap \cD$. Consider a sequence of increasing domains $\seq\sU$ satisfying
$$
\cup_k \sU_k = \cD_1, \quad \sU_k \Subset \cD_1\quad \forall\; k \in \mathbb N,
$$
and each $\sU_k$ is the union of finitely many disjoint smooth open sets. Indeed, such an collection can be
constructed as follows: $\cD_1$ can be written as countable union of connected open sets, and each connected
component can be written as increasing union of smooth open sets. Therefore, we can easily obtain $\sU_k$
by a standard diagonalization procedure.

Fix any $k$ and define
$$
v_k(x)=\Exp^x\left[\int_0^{\uptau_k} f(X_s) ds\right] + \Exp^x[u(X_{\uptau_k})],
$$
where $\uptau_k$ is the first exit time from $\sU_k$. Since the boundary of $\partial\sU_k$ is regular by \cite[Lem.~2.9]{BGR15},
it is routine to check that $v_k\in\cC_{\rm b}(\Rd)$, see e.g. \cite[Sect.~3.1]{B18}. Moreover, $v_k$ is a viscosity solution
(see \cite{B18} , \cite[Lem.~3.7]{KKLL}) of
$$
-\Psidel v_k = - f\quad \text{in}\; \sU_k, \quad \text{and}\quad v_k= u\quad \text{in}\; \; \sU_k^c.
$$
Thus by the comparison principle in \cite[Th.~3.8]{KKLL} we then have $u\leq v_k$ in $\Rd$. Using \cite[Th.~3.1]{BL17b}
we obtain a constant $C$, dependent on $\Psi, p, \diam \cD$, satisfying
$$
\sup_{\sU_k} u^+\leq \sup_{\sU_k} v^+\leq \sup_{\sU^c_k} u^+ + C \norm{f}_{L^p(\cD)}.
$$
Letting $k\to\infty$, we finally obtain \eqref{ET3.1A}.
\end{proof}

\begin{remark}
We note that an estimate similar to \eqref{ET3.1A} has also been obtained in \cite[Prop.~1.4]{ROS} for fractional
Laplacian operators. In this paper the authors considered solutions instead of sub-solutions, and their proof
technique used an explicit formula giving the Green function of $(-\Delta)^{\nicefrac{\alpha}{2}}$ in $\Rd$.
\end{remark}

As a consequence we note the following result of its own interest for viscosity solutions.
\begin{corollary}[\textbf{Maximum principle for narrow domains}]
\label{C3.1}
Suppose that $\Psi$ satisfies WLSC. Let $u\in\cC_{\rm b}(\Rd)$ be a viscosity sub-solution of
$$
-\Psidel u + c\, u= 0\quad \text{in}\;  \cD, \quad \text{and}\quad u\leq 0\quad \text{in}\; \; \cD^c.
$$
There exists $\varepsilon = \varepsilon(\Psi,\diam \cD, \norm{c}_\infty)>0$ such that $u\leq 0$ in $\Rd$, whenever 
$|\cD|<\varepsilon$.
\end{corollary}

\begin{proof}
The result follows from \cref{T3.1} by choosing $f(x)=\norm{c}_\infty u^+(x)$ on $\{u>0\}$.
\end{proof}

Let $\cC_{{\rm b}, +}(\cD)$ be the space of bounded non-negative functions on $\Rd$ that are positive in $\cD$.
Also, recall that $c:\cD\to\RR$ is a bounded continuous function. To proceed to our next result, define the principal 
eigenvalue of $-\Psidel + c$ as
$$
\lambda_\cD=\inf\big\{\lambda \;  :\; \exists \; \psi\in \cC_{{\rm b}, +}(\cD)\;\; \text{such that}\, -\Psidel\psi +
c\, \psi\leq \lambda \psi\; \text{in}\; \cD\big\}.
$$
The above is in the spirit of \cite{BNV}, see also \cite{Armstrong, B18b}.
In what follows, we will be interested in characterizing the principal eigenfunction of $-\Psidel + c$ in $\cD$.

\begin{theorem}\label{T2.1}
Suppose that $\Psi$ satisfies the WLSC property with parameters $(\underline{\mu}, \underline{c}, \underline{\theta})$. 
There exists a unique $\varphi_\cD\in \cC_{{\rm b}, +}(\cD)$ with
$\varphi_\cD(0)=1$, satisfying
\begin{equation}\label{ET2.1A}
\left\{\begin{array}{ll}
-\Psidel \varphi_\cD + c\, \varphi_\cD = \lambda_\cD\, \varphi_\cD \quad \mbox{in} \; \cD,
\\[2mm]
\hspace{2.8cm} \varphi_\cD = 0 \qquad\quad \mbox{in}\; \cD^c.
\end{array}
\right.
\end{equation}
%\begin{equation}
%\begin{split}
%-\Psidel \varphi_\cD(x) + c(x) \varphi_\cD(x) &= \lambda_\cD\, \varphi_\cD(x) \quad \text{in} \; \cD
%\\
%\varphi_\cD &= 0 \quad \text{in}\; \cD^c.
%\end{split}
%\end{equation}
\end{theorem}
\begin{proof}
First note that it follows from the arguments of \cite[Rem.~3.2]{BL17b} that $\psi^*$ in \eqref{E3.2} is a
viscosity solution of
\begin{equation*}\label{ET2.1B}
-\Psidel \psi^* + c \psi^* = \lambda\psi^*\quad \text{in}\; \cD,\quad \text{and}\quad \psi^*=0\quad \text{in}\; \cD^c.
\end{equation*}
We show that $\lambda=\lambda_\cD$. It follows from the definition that $\lambda\geq\lambda_\cD$. Suppose that
$\lambda>\lambda_\cD$. Then there exist $\gamma<\lambda$ and $\psi\in\cC_{{\rm b}, +}(\cD)$ such that
$$
-\Psidel\psi + c\psi\leq \gamma \psi \quad \text{in}\; \cD.
$$
Let $w_z (x) = z\psi^*(x)-\psi(x)$, $z\in\RR$. Fix a compact set $\cK\Subset \cD$ such that $|\cK^c\cap \cD|<\varepsilon$, 
where $\varepsilon$ is the same as in \cref{C3.1}. Take
$$
\mathfrak{z}=\sup\{z>0\; :\; w_z <0 \;\; \text{in}\; \cD\}.
$$
Since $\psi^*>0$ in $\cD$, it is easily seen that $\mathfrak{z}<\infty$. We claim that $\mathfrak{z}>0$. Indeed, note that 
by \cite[Lem.~5.8]{CS09} we have for every $z>0$
\begin{equation}\label{ET2.1C}
-\Psidel w_z +  (c-\lambda) w_z\geq (\lambda-\gamma)\psi > 0 \;\; \text{in}\; \cD.
\end{equation}
Since $\psi>0$ in $\cD$, we can choose $z$ small enough so that $w_z<0$ in $\cK$. Thus by \cref{C3.1} we have
$w_z\leq 0$ in $\cD$. Next suppose that $w_z(x_0)=0$, for a suitable $x_0\in \cD$. Consider a non-positive test
function $\varphi\in \cC_{\rm b}(x_0)$ above $w_z$ such that $\varphi(y)=0$ in $\sB_\delta(x_0)\subset \cD$ and $\varphi(y)=
w_z(y)$ in $\sB_{2\delta}(x_0)$. Applying the definition of viscosity sub-solution to \eqref{ET2.1C} we see that
$$
-\Psidel \varphi(x_0)\geq 0
$$
which implies
$$
\int_{\Rd} \varphi (x_0+y) j(|y|) dy=0.
$$
Since $\delta$ can be chosen arbitrarily small, this implies $w_z=0$ in $\Rd$, which contradicts the fact that $w_z<0$
in $\cK$. Thus $w_z<0$ in $\cD$ follows, and hence we get $\mathfrak{z}>0$. Moreover, by a similar argument we can
also show that either $w_{\mathfrak{z}}=0$ in $\Rd$ or $w_{\mathfrak{z}}<0$ in $\cD$. Note that the first case can
be ruled out since $\gamma<\lambda$. On the other hand, if $w_{\mathfrak{z}}<0$ in $\cD$, we can choose $\eta>0$ such
that $w_{\mathfrak{z} + \eta}<0$ in $\cK$ and a similar argument as above gives $w_{\mathfrak{z} + \eta}<0$ in $\cD$,
in contradiction with the definition of $\mathfrak{z}$. Thus no $\gamma$ exists and hence $\lambda=\lambda_\cD$.

To prove uniqueness, we need to show that for every $\psi\in\cC_{{\rm b}, +}(\cD)$ satisfying
$$
-\Psidel \psi + c \psi\leq \lambda_\cD \psi \quad \text{in}\; \cD,\quad \text{and}\quad \psi=0\quad \text{in}\; \cD^c,
$$
there exists $\kappa > 0$ such that $\kappa\psi=\psi^*$. This follows directly from the argument above.
\end{proof}

Our next result is Hopf's lemma for the class of non-local operators we consider. Denote
$$
\delta_\cD(x)=\dist (x, \partial \cD) \quad \mbox{and} \quad \phi(r)= \frac{1}{\sqrt{\Psi(1/r^{2})}}.
$$

Our next result is about the boundary point lemma and strong maximum principle. See also \cite{Ciomaga} for a
strong maximum principle for a class of operators using a different approach.
\begin{theorem}[\textbf{Hopf's Lemma}]
\label{T2.2}
Let $u\in\cC_{\rm b}(\Rd)$ be a non-negative viscosity super-solution of
\begin{equation}\label{ET2.2A}
-\Psidel u + c\, u= 0 \quad \text{in}\; \cD,
\end{equation}
Then either $u=0$ in $\Rd$ or $u>0$ in $\cD$. Furthermore, if Assumption \ref{A2.1} holds and $u>0$ in $\cD$, then
there exists a constant $\eta>0$ such that
\begin{equation}\label{ET2.2B}
\frac{u(x)}{\phi(\delta_\cD(x))} \geq \eta, \quad x\in \cD.
\end{equation}
\end{theorem}

\begin{proof}
With no loss of generality we may assume that $c\in\cC(\bar{\cD})$, else $c$ should be replaced by $-\sup_{\cD}|c|$.
Suppose that $u$ is not positive in $\cD$. Then there exists $x_0\in \cD$ such that $u(x_0)=0$. Consider a non-negative 
test function $\varphi\in\cC_{\rm b}(x_0)$ below $u$ such that $\varphi(y)=0$ for $y\in \sB_\delta(x_0)\subset \cD$, and 
$\varphi(y)=u(y)$ for $y\in \sB^c_{2\delta}(x_0)$, with an arbitrary $\delta > 0$. Since $u$ is a viscosity super-solution 
of \eqref{ET2.2A}, it follows that
$$
-\Psidel \varphi(x_0) + c(x_0) \varphi(x_0)\leq 0,
$$
which implies
$$
\int_{\Rd} \varphi (x_0+y) j(|y|) dy=0.
$$
Since $\delta$ is arbitrary, it follows that $u=0$ in $\Rd$, which proves the first part of the claim.

Now we prove the second part. By the given condition we have $u>0$ in $\cD$. Denote by $v_n$ the solution of
\begin{equation*}\label{ET2.2C}
-\Psidel v_n =-\frac{1}{n}\quad \text{in} \; \cD, \quad \text{and}\quad v_n=0\quad \text{in}\; \cD^c.
\end{equation*}
As it is well known, see \cite{KKLL}, $v_n(x)=\frac{1}{n}\Exp^x[\uptau_\cD]$. We claim that for a large enough $n$
we have
\begin{equation}\label{ET2.2D}
u(x)\geq v_n(x) \quad \text{for}\; x\in\Rd.
\end{equation}
Note that $w_n(x)=u(x)-v_n(x)\geq 0$ in $\cD^c$. Assume, to the contrary, that \eqref{ET2.2D} does not hold. Then
there exists a sequence $\seq x$ such that
$$
w_n(x_n)=\min_{\Rd} w_n <0,
$$
and since $v_n\to 0$ uniformly in $\Rd$, necessarily $x_n\to\partial \cD$ as $n\to\infty$. Let $\cK$ be a nonempty
compact subset of $\cD$ and denote $M=\min_{x\in \cK} u(x)>0$. Choose $n$ large enough so that $x_n\notin \cK$. Note
that $\varphi(x)=w_n(x_n)$ crosses $w_n$ from below, and by \cite[Lem.~5.8]{CS09}
$$
-\Psidel w_n \leq -c u+ \frac{1}{n}\quad \text{in}\; \cD\,,
$$
holds. Hence, by using the definition of a viscosity super-solution it is clear that
$$
\int_{x_n+z\in \cK} (w_n(x_n+y)-w_n(x_n)) j(|y|) dy \;\leq -c(x_n) u(x_n) + \frac{1}{n}\xrightarrow{n\to\infty} 0.
$$
However, $\norm{v_n}_\infty\leq \frac{1}{n} \norm{v_1}_\infty$, and therefore,
$$
\int_{x_n+y\in \cK} (w_n(x_n+y)-w_n(x_n))j(|y|) dy\;
\geq (M-\frac{1}{n} \norm{v_1}_\infty) \int_{x_n+y\in\cK} j(|y|) dy>0,
$$
as $n\to\infty$. This proves \eqref{ET2.2D}. Thus \eqref{ET2.2B} follows by a combination of \eqref{ET2.2D}, \eqref{E3.6}
and \eqref{E3.5}.

%Our first observation is the following. If $v\in\cC_b(\Rd)$ is a viscosity solution of
%$$-\Psidel v = f\quad \text{in}\; D$$
%then for any $c\in \cC(\bar D)$ we have
%\begin{equation}\label{ET2.2C}
%\Exp^x\left[e^{\int_0^{\uptau\wedge t} c(X_s) ds} v(X_{\uptau\wedge t})\right]-v(x) = \Exp^x \left[\int_0^{\uptau\wedge t} e^{\int_0^s c(X_p) dp} (c(X_s) v(X_s) + f(X_s)) ds\right],
%\end{equation}
%for all $t\geq 0, x\in D$, where $\uptau$ is the exit time from $D$. Proof of \eqref{ET2.2C} follows from an analogous argument as in \cite[Lemma~3.1]{B18a}.
%Thus using \eqref{ET2.2B} we get
%$$u(x)\geq \Exp^x\left[e^{\int_0^{\uptau\wedge t} c(X_s) ds} u(X_{\uptau\wedge t})\right]
%\geq \Exp^x\left[e^{\int_0^{ t} c(X_s) ds} u(X_{t})\Ind_{\{t<\uptau\}}\right], \quad \forall\; t\geq 0, x\in D.$$
%Consider $D_1\Subset D$. Now fixing $t=2$ we get from above that
%\begin{align*}
%u(x) &\geq e^{-2 \norm{c}_\infty} \Exp^x[u(X_t)\Ind_{\{t<\uptau\}}]
%\\
%&= e^{-2 \norm{c}_\infty} \int_D u(y) p^D(2, x, y) dy
%\\
%&\geq \kappa_1 e^{-2 \norm{c}_\infty} (\min_{D_1} u) \Prob^x(\uptau>1) \int_{D_1} \Prob^y(\uptau>1) p(2\wedge V^2(r), \abs{x-y}) dy
%\\
%&\geq \kappa_1 e^{-2 \norm{c}_\infty} (\min_{D_1} u)\,  p(2\wedge V^2(r), \diam D) \left[\int_{D_1} \Prob^y(\uptau>1) dy\right] \Prob^x(\uptau>1)
%\end{align*}
%where in the third line we use \eqref{E3.6}. Now the proof follows from \eqref{E3.7} and \eqref{E3.5}.
\end{proof}

\begin{remark}
{\rm
Choosing $\Psi(s)=s^{\alpha/2}$, $\alpha\in (0, 2)$, above we get back Hopf's lemma for the fractional Laplacian,
extending \cite{DQ17,FJ15,GS16} to viscosity solutions.
}
\end{remark}

\subsection{Maximum principles}

Now we turn to discussing several maximum principles for viscosity solutions.
\begin{theorem}[\textbf{Refined maximum principle}]
\label{T2.3}
Suppose that $\Psi$ satisfies the WLSC property. Let $\lambda_\cD<0$, and $v\in\cC_{\rm b}(\Rd)$ be a viscosity
sub-solution of
$$
-\Psidel v + c\, v= 0\quad \text{in}\; \cD, \quad v\leq 0\quad \text{in}\; \cD^c.
$$
Then $v\leq 0$ in $\Rd$.
\end{theorem}
\begin{proof}
Define
$$
\mathfrak{z}=\sup\{t>0\; :\; \varphi_\cD-t v>0\quad \text{in}\, \cD\}.
$$
Note that $\mathfrak{z}>0$. Indeed, fix a compact $\cK\Subset\cD$ such that $|\cK^c\cap\cD|<\varepsilon$,
where $\varepsilon$ is the same as in Corollary~\ref{C3.1}. Then choose $t>0$ small enough so that
$w_t=\varphi_\cD-t v$ is positive in $\cK$. Since
$$-\Psidel w_t + c\, w_t\leq \lambda_\cD \varphi_\cD<0\quad \text{in} \; D\setminus\cK,
\quad w_t\geq 0 \quad \text{in}\; (D\setminus\cK)^c,$$
we get from Corollary~\ref{C3.1} and Theorem~\ref{T2.2} that $w_t>0$ in $\cD$. Thus $\mathfrak{z}\geq t>0$.
Now suppose, to the contrary, that $v(x_0)>0$ for some $x_0\in\cD$. Then we would get
$\mathfrak{z}<\infty$. Note that
\begin{equation}\label{ET2.3A}
-\Psidel w_\mathfrak{z} + c\, w_\mathfrak{z}\leq \lambda_\cD \varphi_\cD<0\quad \text{in} \; D,
\quad w_\mathfrak{z}\geq 0 \quad \text{in}\; D^c.
\end{equation}
Using Theorem~\ref{T2.2}, we have either $w_\mathfrak{z}=0$ in $\Rd$ or $w_\mathfrak{z}>0$ in $\cD$. In view
of \eqref{ET2.3A}, the case $w_\mathfrak{z}=0$ in $\Rd$ cannot occur. Again, if $w_\mathfrak{z}>0$ in $\cD$,
by repeating the above argument we can show that $\delta>0$ exists, satisfying $w_{\mathfrak{z}+\delta}>0$
in $\cD$ (see Theorem~\ref{T2.1} for a similar argument). This contradicts the definition of $\mathfrak{z}$.
Hence no such $x_0$ exists, implying $v\leq 0$ in $\Rd$.
\end{proof}

To prove our next main theorem below, we need the following result in the spirit of \cite[Th.~1.2]{B18b}.
\begin{lemma}\label{T3.2}
Let $u\in\cC_{\rm b}(\Rd)$ be a viscosity solution of
\begin{equation}\label{ET3.2A}
-\Psidel u + (c-\lambda_\cD)u=0\quad \text{in}\; \cD, \quad \text{and}\quad u=0\quad \text{in}\; \cD^c,
\end{equation}
or of
\begin{equation}\label{ET3.2B}
-\Psidel u + (c-\lambda_\cD)u\geq 0\quad \text{in}\; \cD,
\quad \text{and} \quad u\leq 0\quad \text{in}\; \cD^c,\quad u(x_0)>0
\end{equation}
for an $x_0\in \cD$. Then we have $u=z\varphi_\cD$ for some $z\in\RR$, where $\varphi_\cD$ is given by \eqref{ET2.1A}.
\end{lemma}

\begin{proof}
We provide a proof considering \eqref{ET3.2B}, while the proof for \eqref{ET3.2A} is analogous. We follow a similar line
of argument as in the proof of \cref{T2.1}. Fix a compact set $\cK\Subset \cD$ such that $|\cK^c\cap \cD|<\varepsilon$,
where $\varepsilon$ is the same as in \cref{C3.1}. Denote $w_z=\varphi_\cD-z u$, for $z>0$. Then $w_z\geq 0$ in $\cD^c$.
Let
$$
\mathfrak{z}=\sup\{z>0\; :\; w_z > 0 \; \text{in}\; \cD\}.
$$
Note that $\mathfrak{z}$ is finite, which follows from the fact that $u(x_0)>0$. As before, we claim that $\mathfrak{z}>0$.
Indeed, note that by \cite[Lem.~5.8]{CS09} we have for every $z>0$
\begin{equation*}
-\Psidel w_z +  (c-\lambda_\cD) w_z\leq 0 \quad \text{in}\; \cD.
\end{equation*}
Then by using a similar argument as in \cref{T2.1}, it is easily seen that $\mathfrak{z}>0$. Note that by \cref{T2.2} either
$w_\mathfrak{z}=0$ in $\Rd$ or $w_\mathfrak{z}>0$ in $\cD$ holds. If $w_\mathfrak{z}>0$, then by following the arguments of
\cref{T2.1} we get a contradiction. Thus $w_\mathfrak{z}=0$ in $\Rd$ and this completes the proof.
\end{proof}

The following result establishes an anti-maximum principle.
\begin{theorem}[\textbf{Anti-maximum principle}]
\label{T2.4}
Let Assumptions \ref{A2.1}-\ref{A2.2} hold, and $f\in \cC(\bar \cD)$ be such that $f\lneq 0$. Then there exists
$\delta>0$ such that for every $\lambda\in (\lambda_\cD-\delta, \lambda_\cD)$, the solution of
\begin{equation*}\label{ET2.4A}
-\Psidel u + (c-\lambda)u=f \quad \text{in}\; \cD, \quad \text{and}\quad u=0\quad \text{in}\; \cD^c,
\end{equation*}
satisfies $u<0$ in $\cD$.
\end{theorem}
\begin{proof}
We prove the theorem by assuming, to the contrary, that no such $\delta$ exists. Then there exists a sequence
$(u_n, \lambda_n)_{n\in\mathbb N}\subset \cC_{\rm b}(\Rd)\times\RR$ such that $u_n\nless 0$, $\lambda_n\to\lambda_\cD$ as
$n\to\infty$, and
\begin{equation}\label{ET2.4B}
-\Psidel u_n + (c-\lambda_n)u_n=f \quad \text{in}\; \cD, \quad \text{and}\quad u_n=0\quad \text{in}\; \cD^c.
\end{equation}
Note that $\liminf_{n\to\infty} \norm{u_n}_\infty>0$; otherwise, we can pass to the limit and obtain $0$ for a viscosity
solution of \eqref{ET2.4B}, contradicting the fact that $f\neq 0$. We split the proof in two steps.

\medskip
\noindent
\emph{Step 1}: Suppose that $\limsup_{n\to\infty} \norm{u_n}_\infty<\infty$. Using \cite[Th.~1.1]{KKLL} it follows that
$$
\sup_{n\in\mathbb N}\, \sup_{x, y\in \cD}\abs{u_n(x)-u_n(y)}\leq \kappa_1 \phi (\abs{x-y}),
$$
for a constant $\kappa_1$, where $\phi(r)= 1/\sqrt{\Psi(1/r^{2})}$. Thus $\seq u$ is equicontinuous and has a
subsequence convergent to a limit $u\neq 0$, which is a solution of
\begin{equation}\label{ET2.4C}
-\Psidel u + (c-\lambda_\cD)u=f \quad \text{in}\; \cD, \quad \text{and}\quad u=0\quad \text{in}\; \cD^c.
\end{equation}
If $u(x_0)<0$ for some $x_0\in \cD$, then it follows from \eqref{ET2.4C} and \cref{T3.2} that $u=z\varphi_\cD$ for some
$z<0$ (as $\varphi_\cD>0$ in $\cD$), contradicting the fact that $f\neq 0$. Thus we have $u\geq 0$ in $\cD$, and by \cref{T2.2}, $u>0$ in $\cD$. Then
the proof of \cref{T2.1} implies that $u=z\varphi_\cD$, again in contradiction with $f\neq 0$.

\medskip
\noindent
\emph{Step 2}: Suppose that $\limsup_{n\to\infty} \norm{u_n}_\infty=\infty$. Define $v_n=u_n/\norm{u_n}_\infty$.
By using a similar argument as in the previous step, we find a subsequence $v_{n_k}\to v\neq 0$ satisfying
\begin{equation}\label{ET2.4D}
-\Psidel v + (c-\lambda_\cD)v=0 \quad \text{in}\; \cD, \quad \text{and}\quad u=0\quad \text{in}\; \cD^c.
\end{equation}
Using \eqref{ET2.4D} and \cref{T3.2} we have $v=z\varphi_\cD$, for some $z\neq 0$. Recalling the renewal function from
\eqref{E3.5} and using \cite[Th.~1.2]{KKLL}, we have
$$
\sup_{x\in \cD}\left|\frac{v_{n_k}(x)}{V(\delta_\cD(x))} -\frac{v(x)}{V(\delta_\cD(x))}\right|\to 0,
\quad \text{as}\;\; n_k\to\infty.
$$
By \cref{T2.2} and \eqref{E3.5} we know that
$$
\inf_{x\in \cD}\frac{\varphi_\cD(x)}{V(\delta_\cD(x))}\geq \eta>0,
$$
and therefore, the above estimates show that
$$
\inf_{x\in \cD}\frac{v_{n_k}(x)}{V(\delta_\cD(x))}\geq \frac{\eta}{2}\, z, \quad \text{if}\;\; z>0,
\quad\text{or}\quad \sup_{x\in \cD}\frac{v_{n_k}(x)}{V(\delta_\cD(x))}\leq \frac{\eta}{2}\, z, \quad \text{if}\;\; z<0.
$$
Note that the second possibility contradicts our hypothesis on the sequence $u_{n_k}\nless 0$, as if the first one were the case,
then $u_{n_k}\in \cC_{\rm b, +}(\Rd)$ and
$$
-\Psidel u_{n_k} + (c-\lambda_{n_k}) u_{n_k}\leq 0, \quad \lambda_{n_k}<\lambda_\cD,
$$
would follow, which is incompatible with the definition of $\lambda_\cD$.
\end{proof}

The arguments of the above proof also give the following result on the principal eigenvalue being isolated.
\begin{theorem}\label{T3.3}
Let Assumptions \ref{A2.1}-\ref{A2.2} hold. Then there exists $\varepsilon>0$ such that there is no non-trivial
solution of
\begin{equation}\label{ET3.3A}
-\Psidel u + c\, u = \mu \, u\quad \text{in}\; \cD, \quad \text{and}\quad u=0\quad \text{in}\; \cD^c,
\end{equation}
for $\mu\in (\lambda_\cD- \varepsilon, \infty)\setminus\{\lambda_\cD\}$.
\end{theorem}

\begin{proof}
Suppose $\mu>\lambda_\cD$. Then the principal eigenvalue of the operator is $-\Psidel + (c-\mu)$ is negative. Hence,
by \cref{T2.3} the Dirichlet problem \eqref{ET3.3A} cannot have any solution other than $0$. Thus we consider $\mu<
\lambda_\cD$ and suppose that no such $\varepsilon$ exists. Then there exists a sequence $(u_n, \mu_n)_{n\in\mathbb N}$
of non-zero solutions such that $\mu_n \uparrow \lambda_\cD$ and
\begin{equation}\label{ET3.3B}
-\Psidel u_n + c\, u_n = \mu_n \, u_n\quad \text{in}\; \cD, \quad \text{and}\quad u_n=0\quad \text{in}\; \cD^c.
\end{equation}
Following the arguments of \cref{T2.4} and using \eqref{ET3.3B}, we see that there exists $u\in\cC_{\rm b}(\Rd)$ with
$\norm{u}_\infty=1$ satisfying
\begin{equation*}
-\Psidel u + c\, u = \lambda_\cD \, u\quad \text{in}\; \cD, \quad \text{and}\quad u=0\quad \text{in}\; \cD^c.
\end{equation*}
As before, necessarily we have that $u=z\varphi_\cD$ for some $z\neq 0$. Applying the arguments of \cref{T2.4} again,
we can show that some of the $u_n$ in \eqref{ET3.3B} are positive in $\cD$, contradicting the definition of $\lambda_\cD$.
\end{proof}

\section{\bf Applications}
\subsection{Rotational symmetry of positive solutions}
In classical PDE theory maximum principles proved to be useful in establishing symmetry properties of solutions. Next we
show that our narrow domain maximum principle \cref{C3.1} can be used to establish radial symmetry of the positive solutions
in rotationally symmetric domains. The main result of this section is the following.

\begin{theorem}\label{T3.4}
Let $\Psi$ satisfy the WLSC property with parameters $(\underline{\mu}, \underline{c}, \underline{\theta})$. Suppose that $\cD$
is convex in the direction of the $x_1$ axis, and symmetric about
the plane $\{x_1=0\}$. Also, let $f:[0, \infty)\to\RR$ be locally Lipschitz continuous, and $g:\cD\to\RR$ be a
symmetric function with respect to $x_1=0$ and decreasing in the $x_1$ direction. Consider a solution of
\begin{equation}\label{ET3.4A}
\left\{\begin{array}{lll}
-\Psidel u = f(u) - g(x) \quad \text{in}\; \cD,
\\[2mm]
\hspace{1.5cm} u > 0 \quad \text{in}\; \cD,
\\[2mm]
\hspace{1.5cm} u =0 \quad \text{in}\; \cD^c.
\end{array}
\right.
\end{equation}
%\begin{equation}
%\begin{split}
%-\Psidel u &= f(u) - g(x) \quad \text{in}\; \cD,
%\\
%u &> 0 \quad \text{in}\; \cD,
%\\
%u& =0 \quad \text{in}\; \cD^c.
%\end{split}
%\end{equation}
Then $u$ is symmetric with respect to $x_1=0$ and strictly decreasing in the $x_1$ direction.
\end{theorem}

\begin{proof}
Part of the proof is standard and we only sketch the main steps involved; for notations and some details we keep to
\cite[Th.~1.1]{FW14}. Define
\begin{align*}
\Sigma_\lambda = \{x=(x_1, x')\in \cD\; :\; x_1>\lambda\}
\quad & \mbox{and} \quad T_\lambda = \{x=(x_1, x')\in \Rd\; :\; x_1=\lambda\}, \\
u_\lambda(x)=u(x_\lambda)\quad & \text{and}\quad w_\lambda= u_\lambda(x)-u(x),
\end{align*}
where $x_\lambda=(2\lambda-x_1, x')$. For a set $A$ we denote by $\sR_\lambda A$ the reflection of $A$ with respect
to the plane $T_\lambda$. Also, define
$$
\lambda_{\max}=\sup\{\lambda>0 \; :\; \Sigma_\lambda\neq \emptyset\}.
$$
We note that for any $\lambda \in (0, \lambda_{\max})$, $u_\lambda$ is a viscosity solution of
$$
-\Psidel u_\lambda = f(u_\lambda) - g(x_\lambda) \quad \text{in}\; \Sigma_\lambda,
$$
and thus by \cite[Lem.~5.8]{CS09} we obtain from \eqref{ET3.4A} that
\begin{equation}\label{ET3.4B}
-\Psidel w_\lambda = f(u_\lambda)-f(u) + g(x) - g(x_\lambda)\quad \text{in}\; \Sigma_\lambda.
\end{equation}
Define $\Sigma^-_\lambda=\{x\in\Sigma_\lambda\; :\; w_\lambda<0\}$. Since $w_\lambda\geq0$ on $\partial\Sigma_\lambda$, it follows that $w_\lambda=0$ on $\partial\Sigma_\lambda^-$. Hence the function
\[
v_\lambda=\left\{\begin{array}{lll}
w_\lambda & \text{in}\; \Sigma_\lambda^-,
\\[2mm]
0 & \text{elsewhere},
\end{array}
\right.
\]
is in $\cC_{\rm b}(\Rd)$. We claim that for every $\lambda\in (0, \lambda_{\max})$
\begin{equation}\label{ET3.4C}
-\Psidel v_\lambda \leq f(u_\lambda)-f(u) + g(x) - g(x_\lambda)\quad \text{in}\; \Sigma^-_\lambda.
\end{equation}
To see this, let $\varphi$ be a test function that crosses $v_\lambda$ from below at a point $x\in \Sigma^-_\lambda$.
Then we see that $\varphi + (w_\lambda-v_\lambda)\in \cC_{\rm b}(x)$ and crosses $w_\lambda$ at $x$ from below. Denote
$\zeta_\lambda(x)=w_\lambda-v_\lambda$. Using \eqref{ET3.4B} it follows that
\begin{equation}\label{ET3.4D}
-\Psidel (\varphi+\zeta_\lambda)(x) \leq f(u_\lambda(x))-f(u(x)) + g(x) - g(x_\lambda).
\end{equation}
To obtain \eqref{ET3.4C} from \eqref{ET3.4D} we only need to show that
$$
\int_{\Rd} (\zeta_\lambda(x+z)-\zeta_\lambda(x)) j(\abs{z}) dz\geq 0.
$$
This can be done by following the argument of \cite[p8]{FW14} combined with the fact that $j:(0, \infty)
\to (0, \infty)$ is a strictly decreasing function. The proof can then be completed by the standard method of
moving planes.

\medskip
\noindent
\emph{Step 1}: If $\lambda<\lambda_{\max}$ is sufficiently close to $\lambda_{\max}$, then $w_\lambda>0$ in
$\Sigma_\lambda$. Indeed, note that if $\Sigma^-_\lambda\neq\emptyset$, then $v_\lambda$ satisfies \eqref{ET3.4C}.
Denoting
$$
c(x)=\frac{f(u_\lambda(x))-f(u(x))}{u_\lambda(x)-u(x)},
$$
and using the property of $g$, it then follows that
$$
-\Psidel v_\lambda - c(x) v_\lambda\leq 0 \quad \text{in}\; \Sigma^-_\lambda.
$$
Thus choosing $\lambda$ sufficiently close to $\lambda_{\max}$, it follows from \cref{C3.1} that $v_\lambda\geq 0$
in $\Rd$. Hence $\Sigma^-_\lambda=\emptyset$ and we have a contradiction. To show that $w_\lambda>0$ in $\Sigma_\lambda$,
assume to the contrary that $w_\lambda(x_0)=0$ for some $x_0\in\Sigma_\lambda$. Consider a non-negative test function $\varphi\in
\cC_{\rm b}(x_0)$, crossing $w_\lambda$ from below, with the property that $\varphi=0$ in $\sB_r(x_0)\Subset \Sigma_\lambda$
and $\varphi=w_\lambda$ in $\sB_{2r}(x_0)$. Furthermore, choose $r$ small enough such that $\sB_{2r}(x_0)\Subset \Sigma_\lambda$
and $\varphi\geq 0$ in $\Sigma_\lambda$. Then by using \eqref{ET3.4B} we obtain
\begin{equation}\label{ET3.4E}
-\Psidel \varphi(x_0)\leq  g(x_0) - g((x_0)_\lambda)\leq 0.
\end{equation}
Next we compute $-\Psidel \varphi(x_0)$. Note that $\varphi\geq 0$ in $R_\lambda=\{x\in\Rd\; :\; x_1\geq \lambda\}$. We have
\begin{align*}
-\Psidel \varphi(x_0) &= \int_{\Rd} \varphi(z) j(|z-x|) dz
\\
&= \int_{R_\lambda} \varphi(z) j(|z-x_0|) dz + \int_{\sR_\lambda R_\lambda} \varphi(z) j(|z-x_0|) dz
\\
&= \int_{R_\lambda} \varphi(z) j(|z-x_0|) dz + \int_{\sR_\lambda R_\lambda} w_\lambda(z) j(|z-x_0|) dz
\\
&= \int_{R_\lambda} \varphi(z) j(|z-x_0|) dz + \int_{R_\lambda} w_\lambda(z_\lambda) j(|z_\lambda-x_0|) dz
\\
&= \int_{R_\lambda} \varphi(z) j(|z-x_0|) dz - \int_{R_\lambda} w_\lambda(z) j(|z_\lambda-x_0|) dz
\\
&= \int_{R_\lambda\setminus B_{2r}(x_0)} w_\lambda(z) (j(|z-x_0|)-j(|z_\lambda-x_0|)) dz - \int_{B_{2r}(x_0)}
w_\lambda(z) j(|z_\lambda-x_0|) dz.
\end{align*}
Since $|z_\lambda-x_0|>|z-x_0|$ and thus $j(|z-x_0|>j(|z_\lambda-x_0|)$, the first term in the above expression is
non-negative. In fact, since we can choose $r$ arbitrarily small, the first term is positive, unless $w_\lambda =0$ in
$R_\lambda$, contradicting that $w_\lambda\neq 0$ on $\partial\Sigma_\lambda\cap \bar\cD$. Thus the first integral is positive
for some $\hat{r}>0$, and by monotone convergence we obtain
$$
\lim_{r\to 0}\int_{R_\lambda\setminus B_{2r}(x_0)} w_\lambda(z) (j(|z-x_0|)-j(|z_\lambda-x_0|))
\geq \int_{R_\lambda\setminus B_{2\hat{r}}(x_0)} w_\lambda(z) (j(|z-x_0|)-j(|z_\lambda-x_0|)) dz>0.
$$
On the other hand,
$$
\lim_{r\to 0} \int_{B_{2r}(x_0)} w_\lambda(z) j(|z_\lambda-x_0|) dz=0.
$$
Hence there exists $r>0$ small enough such that $-\Psidel \varphi(x_0)>0$, in contradiction with \eqref{ET3.4E}. This proves
that $w_\lambda>0$ is in $\Sigma_\lambda$, and shows the claim of Step 1.

\medskip
\noindent
\emph{Step 2}: It remains to show that $\inf\{\lambda>0\; :\; w_\lambda>0\; \text{in}\; \Sigma_\lambda\}=0$. This actually
follows by using the estimates in Step 1 above, in a similar way as discussed in \cite[p10]{FW14}. Also, strict monotonicity
of $u$ in the $x_1$ direction can be obtained by following the calculations in Step~3 of the same argument.
\end{proof}

\begin{remark}
In Theorem~\ref{T3.4}, the positivity assumption on $u$ in $\cD$ can be relaxed by assuming $u\gneq 0$ in $\cD$ and the same 
conclusion holds. This modification can be adjusted by following through the proof of \cite[Th.~1.1]{JW16} (see claim 1 there); 
see also \cite{BJ20,JW14}.
\end{remark}

Using the radial symmetry of the function $j$ in \eqref{E2.3}, we easily arrive at
\begin{corollary}
Let $\Psi$ satisfy the WLSC property, and $g$ be a radially decreasing function. Then every solution of
\begin{equation*}
\left\{\begin{array}{lll}
-\Psidel u = f(u) - g(x) \quad \text{in}\;\; \sB_1(0),
\\[2mm]
\qquad \qquad u > 0 \quad \text{in}\;\; \sB_1(0),
\\[2mm]
\qquad \qquad  u =0 \quad \text{in}\;\; \sB^c_1(0),
\end{array}
\right.
\end{equation*}
%\begin{equation*}
%\begin{split}
%-\Psidel u &= f(u) - g(x) \quad \text{in}\; \sB_1(0),
%\\
%u &> 0 \quad \text{in}\; \sB_1(0),
%\\
%u& =0 \quad \text{in}\; \sB^c_1(0),
%\end{split}
%\end{equation*}
is radial and strictly decreasing in $\abs{x}$.
\end{corollary}

\begin{remark}
{\rm
By similar arguments as in \cite[Th.~1.3]{FW14}, we can extend our result to the following system of equations, and
establish radial symmetry of the positive solutions of
\begin{equation*}
\left\{\begin{array}{lll}
-\Psidel u = f_1(v) - g_1(x) \quad \text{in}\; \sB_1(0),
\\[2mm]
-\Psidel v = f_2(u) - g_2(x) \quad \text{in}\; \sB_1(0),
\\[2mm]
\qquad \qquad u>0,\; v>0 \;  \quad \text{in}\; \sB_1(0),
\end{array}
\right.
\end{equation*}
%\begin{align*}
%-\Psidel u &= f_1(v) - g_1(x) \quad \text{in}\; \sB_1(0),
%\\
%-\Psidel v &= f_2(u) - g_2(x) \quad \text{in}\; \sB_1(0)
%\\
%u>0,\;& v>0 \;  \quad \text{in}\; \sB_1(0),
%\end{align*}
where $f_1, f_2$ are locally Lipschitz continuous and decreasing, and $g_1, g_2$ are radially decreasing.
}
\end{remark}

\subsection{The overdetermined non-local torsion equation}
In this section we use our maximum principle to revisit the overdetermined torsion problem considered in
\cite[Th.~1.3]{GS16}. Denote
$$
\phi(r) =\frac{1}{\sqrt{\Psi(1/r^2)}}, \quad r>0.
$$
As seen in \cref{T2.2}, the function $\phi$ describes the boundary behaviour of the Dirichlet solutions. Also, recall 
the renewal function $V$ from \eqref{E3.5}-\eqref{E3.6}. When $\Psi$ is regularly varying at infinity with some parameter 
$\alpha>0$, we know from \cite[Prop.~4.3, Rem.~4.7]{KMR} that
\begin{equation}\label{E4.6}
\lim_{r\to 0} \frac{V(r)}{\phi(r)}=\kappa>0,
\end{equation}
for a constant $\kappa$. In fact, $\kappa=\frac{1}{\Gamma(1+2\alpha)}$. Now consider a solution $u$ of the non-local 
torsion equation
\begin{equation}\label{E4.7}
-\Psidel u = -1\quad \text{in}\; \cD, \quad \text{and}\quad u=0\quad \text{in}\; \cD^c.
\end{equation}
Also, let Assumptions \ref{A2.1}-\ref{A2.2} hold. Then it is known from \cite{KKLL} that $u(x)/V(\delta_\cD(x))$
is uniformly continuous in $\cD$ and thus it can be extended to $\bar \cD$. Define
$$
\Tr \left(\frac{u}{\phi}\right)(x)=\lim_{\cD\ni z\to x}\frac{u(z)}{\phi(\delta_\cD(z))}, \quad x\in\partial \cD,
$$
thinking of it as the ``trace" on the boundary of the domain. This can also be seen as the non-local analogue of the 
boundary normal derivative for the case $-\Psidel=\Delta$. In view of \eqref{E4.6}, the above map is well-defined
and we have
$$
\Tr \left(\frac{u}{\phi}\right) = \kappa \Tr \left(\frac{u}{V}\right).
$$
Consider the solution $u_r$ of \eqref{E4.7} in a ball $\sB_r(0)$. In particular, $u_r(x)=\Exp^x[\uptau_r]$ where
$\uptau_r$ is the first exit time of $\pro X$ from $\sB_r(0)$. It is immediate that $u_r$ is a radial function and so
the trace is constant on $\{|x|=r\}$. Let $\sH(r)$ be the value of this trace on $\{|x|=r\}$. It is also direct
to see that $\sH(r)$ is non-decreasing. The following result improves on this.

\begin{lemma}
The function $\sH$ is strictly increasing on $(0, \infty)$.
\end{lemma}
\begin{proof}
Consider $0<r<R$. Let $\sB_2$ be the ball of radius $R$ centered at $0$, and $\sB_1$ be a ball of radius $r$ tangential
to $z=(R, 0, \ldots, 0)\in \partial \sB_2$ from inside. 
%(see \cref{Fig. 1}). 
Also, consider a ball $\sB_0$ compactly contained inside $\sB_2\setminus \sB_1$.
%\begin{figure}
%\centering
%\def\svgwidth{0.2\columnwidth}
%\input{plot.pdf_tex}
%\caption{}\label{Fig. 1}
%\end{figure}

Denote by $\uptau_1$ and $\uptau_2$ the first exit times of $\pro X$ from $\sB_1$ and $\sB_2$, respectively. Note that
\begin{equation}\label{EL4.1A}
\sH(r) =\lim_{x\to R}\frac{\Exp^{(x, 0')}[\uptau_1]}{\phi(R-x)} \quad \text{and}\quad \sH(R) =
\lim_{x\to R}\frac{\Exp^{(x, 0')}[\uptau_2]}{\phi(R-x)}.
\end{equation}
Using the strong Markov property of subordinate Brownian motion, we see that for every $x\in \sB_1$
\begin{align*}
\Exp^x[\uptau_2] & = \Exp^x[\uptau_1] + \Exp^x[\uptau_2-\uptau_1]
\\
& = \Exp^x[\uptau_1] + \Exp^x\left[\Ind_{\sB_2\setminus \sB_1}(X_{\uptau_1}) \Exp^{X_{\uptau_1}}[\uptau_2]\right]
\\
&\geq \Exp^x[\uptau_1] + \Exp^x\left[\Ind_{\sB_0}(X_{\uptau_1}) \Exp^{X_{\uptau_1}}[\uptau_2]\right]
\\
&\geq \Exp^x[\uptau_1] + \left(\min_{z\in \sB_0} \Exp^z[\uptau_2]\right) \Prob^x(X_{\uptau_1}\in \sB_0).
\end{align*}
By using the Ikeda-Watanabe formula \cite{IW}, we find that
\begin{align*}
\Prob^x(X_{\uptau_1}\in \sB_0)
&=
\int_{\sB_1} \int_{\sB_0} j(\abs{z-y}) dz G_{\sB_1}(x, dy)
\\
&\geq j(R)\, \abs{\sB_0}\, \int_{\sB_1}  G_{\sB_1}(x, dy)
\\
&=
j(R)\, \abs{\sB_0}\, \Exp^x[\uptau_1],
\end{align*}
where $G_{\sB_1}$ denotes the Green function in $\sB_1$. Thus for a positive $\kappa_1$ we have
$$
\Exp^x[\uptau_2]\geq (1+\kappa_1) \Exp^x[\uptau_1].
$$
Now the proof follows from \eqref{EL4.1A}.
\end{proof}

\begin{remark}
{\rm
Unfortunately, we are not able to find an explicit formula for $\sH$ using $\phi$ or $V$. However, using Assumption
\ref{A2.1} and \cite[Th.~4.1]{BGR15} it is easily seen that $\sH\asymp V$, and therefore, by using \eqref{E3.5} we
get $\sH\asymp \phi$. Note that for $\Psi(s)=s^{\alpha/2}$, $\alpha\in (0, 2)$, the exact expression of the expected
first exit time is known and one can explicitly calculate $\sH$ in this case (see for instance, \cite{GS16}).
}
\end{remark}

Finally, we consider the over-determined torsion problem.
\begin{theorem}\label{T4.2}
Let $\cD$ be a $\cC^{1,1}$ domain containing $0$. Suppose that Assumptions \ref{A2.1}-\ref{A2.2} hold, and $\Psi$ is
regularly varying at infinity. Let $q:(0, \infty)\to(0, \infty)$ be such that $\frac{q}{\sH}$ is non-decreasing in
$(0, \infty)$. Then the overdetermined problem
\begin{equation}\label{ET4.2A}
\left\{\begin{array}{lll}
\quad -\Psidel u = \; -1\qquad \text{in}\;\; \cD,
\\[2mm]
\hspace{1.85cm} u = \; 0 \qquad\;\;\; \text{in}\;\; \cD^c,
\\[2mm]
\hspace{0.9cm} \Tr\left(\frac{u}{\phi}\right) = \; q(\abs{\cdot})\quad \text{on} \;\; \partial \cD,
\end{array}
\right.
\end{equation}
%\begin{equation}
%\begin{split}
%-\Psidel u &= -1\quad \text{in}\; \cD,
%\\
%u & =0 \quad \text{in}\; \cD^c,
%\\
%\Tr(u\phi^{-1})(z)& = q(\abs{z})\quad \text{for}\; z\in\partial \cD,
%\end{split}
%\end{equation}
has a solution if and only if $\cD$ is a ball centered at $0$ and $q=\sH$ on $\partial \cD$.
\end{theorem}
\begin{proof}
From the above discussion we see that \eqref{ET4.2A} always has a solution when $\cD$ is a ball around $0$ and $q=\sH$.
Thus we only need to prove the converse direction. Suppose that $\cD$ is not a ball centered at $0$. Then we can find
two concentric balls $\sB_{r}(0)$ and $\sB_R(0)$, with $r<R$, such that $\sB_r(0)$ touches $z_1\in \partial \cD$ from
inside, and $\cD\subset \sB_R(0)$ with $z_2\in\partial \cD\cap\partial \sB_R(0)$. It is also obvious that $z_1\neq z_2$,
since $r<R$. Let $u_r(x)=\Exp^x[\uptau_r]$, where $\uptau_r$ is the first exit time from $\sB_r(0)$.  Then we have
\begin{equation}\label{ET4.2B}
-\Psidel u_r =-1 \quad \text{in}\; \sB_r(0), \quad \text{and}\quad u_r=0\quad \text{in}\; \sB^c_r(0).
\end{equation}
Similarly, we define $u_R$ in $\sB_R(0)$. Using \cite[Lem.~5.8]{CS09}, we note from \eqref{ET4.2A}-\eqref{ET4.2B} that
\begin{equation}\label{ET4.2C}
-\Psidel (u-u_r) =0\quad \text{in}\; \sB_r(0), \quad \text{and}\quad u-u_r\geq 0\quad \text{in} \; \sB^c_r(0).
\end{equation}
By the comparison principle \cite[Th.~3.8]{KKLL}, we have $u\geq u_r$ in $\Rd$. Similarly, we also have $u\leq u_R$ in
$\Rd$. A combination of this then gives
\begin{equation*}\label{ET4.2D}
u_r\leq u\leq u_R\quad \text{in}\; \Rd.
\end{equation*}
Using \eqref{ET4.2C} and \cref{T2.2}, it follows that either $u=u_r$ in $\Rd$ or $q(|z_1|)>\sH(|z_1|)$. Assuming that the 
first case holds implies $\cD=B_r(0)$, contradicting the assumption. Hence $q(|z_1|)>\sH(|z_1|)$, and a similar argument 
also shows $q(|z_2|)<\sH(|z_2|)$. In sum we have
$$
1<\frac{q(|z_1|)}{\sH(|z_1|)}\leq \frac{q(|z_2|)}{\sH(|z_2|)}<1,
$$
which is impossible, and thus $\cD$ is a ball centered at $0$.  To complete the proof, note that the first two equations in
\eqref{ET4.2A} imply $u(x)=\Exp^x[\uptau_\cD]$, and since $\cD$ is a ball, it follows that $q=\sH$ on $\partial\cD$.
\end{proof}

\subsection*{Acknowledgments}
The authors wish thank the anonymous referees for their careful reading of the manuscript, and helpful comments and suggestions.
This research of AB was supported in part by an INSPIRE faculty fellowship and DST-SERB grants EMR/2016/004810, MTR/2018/000028.


\begin{thebibliography}{777}

\bibitem{A15}
V. Ambrosio:
Periodic solutions for a pseudo-relativistic Schr\"odinger equation, \emph{Nonlinear Anal.} \textbf{120}, 262--284, 2015

\bibitem{A19}
V. Ambrosio:
Concentrating solutions for a class of nonlinear fractional Schr\"odinger equations in ${\mathbb R}^N$,
\emph{Rev. Mat. Iberoam.} \textbf{35}, 1367--1414, 2019

\bibitem{Armstrong}
S.N. Armstrong:
Principal eigenva0lues and an anti-maximum principle for homogeneous fully nonlinear elliptic equations,
\emph{J. Differential Equations} \textbf{246}, 2958--2987, 2009

\bibitem{AL}
G. Ascione and J. L\H{o}rinczi:
Potentials for non-local Schr\"odinger operators with zero eigenvalues, \emph{arXiv:2005.13881}, 2020

\bibitem{B}
J. Bertoin:
\emph{L\'evy Processes}, Cambridge University Press, 1996

\bibitem{BNV}
H. Berestycki, L. Nirenberg and S. R. S. Varadhan:
The principal eigenvalue and maximum principle for second-order elliptic operators in general domains,
\emph{Commun. Pure Appl. Math.} \textbf{47}, 47--92, 1994

\bibitem{BJ20}
A. Biswas and S. Jarohs:
On overdetermined problems for a general class of nonlocal operators,
\emph{J. Differential Equations} \textbf{268}, 2368--2393, 2020

\bibitem{BL17a}
A. Biswas and J. L\H{o}rinczi:
Universal constraints on the location of extrema of eigenfunctions of non-local Schr\"odinger operators,
\emph{J. Differential Equations} \textbf{267}, 267--306, 2019

\bibitem{BL17b}
A. Biswas and J. L\H{o}rinczi:
Maximum principles and Aleksandrov-Bakelman-Pucci type estimates for non-local Schr\"{o}dinger equations with
exterior conditions, \emph{SIAM J. Math. Anal.} \textbf{51}, 1543--1581, 2019


\bibitem{BL20}
A. Biswas and J. L\H{o}rinczi:
Ambrosetti-Prodi type results for Dirichlet problems of fractional-like Laplace operators,
\emph{Integr. Equat. Oper. Th.} \textbf{92}, 26, 2020, https://doi.org/10.1007/s00020-020-02584-7



\bibitem{B18}
A. Biswas:
Liouville type results for system of equations involving fractional Laplacian in the exterior domain,
\emph{Nonlinearity} \textbf{32}, 2246--2268, 2019
%\bibitem{B18a}
%A. Biswas:
%Existence and non-existence results for a class of semilinear nonlocal operators with exterior condition, Preprint. arXiv: 1805.01293, 2018

\bibitem{B18b}
A. Biswas:
Principal eigenvalues of a class of nonlinear integro-differetial operators,
\emph{J. Differential Equations} \textbf{268} (9), 5257--5282, 2020

\bibitem{BGR14}
K. Bogdan, T. Grzywny and M. Ryznar:
Dirichlet heat kernel for unimodal L\'evy processes, \emph{Stochastic Process. Appl.} \textbf{124},
3612--3650, 2014

\bibitem{BGR14b}
K. Bogdan, T. Grzywny and M. Ryznar:
Density and tails of unimodal convolution semigroups, \emph{J. Funct. Anal.} \textbf{266}, 3543--3571, 2014

\bibitem{BGR15}
K. Bogdan, T. Grzywny and M. Ryznar:
Barriers, exit time and survival probability for unimodal L\'evy processes, \emph{Probab. Theory Related Fields}
\textbf{162}, 155--198, 2015

\bibitem{CS14}
X. Cabr\'e and Y. Sire:
Nonlinear equations for fractional Laplacians I: Regularity, maximum principles, and Hamiltonian estimates,
\emph{Ann. Inst. Henri Poincar\'e (C) Nonlin. Anal.}, \textbf{31}, 23--53, 2014

\bibitem{CS15}
X. Cabr\'e and Y. Sire:
Nonlinear equations for fractional Laplacians II: Existence, uniqueness, and qualitative properties of solutions,
\emph{Trans. Amer. Math. Soc.} \textbf{367}, 911--941, 2015


\bibitem{CS09}
L. Caffarelli and L. Silvestre:
Regularity theory for fully nonlinear integro-differential equations,
\emph{Commun. Pure Appl. Math.} \textbf{62}, 597--638, 2009

%\bibitem{CKS}
%Z.-Q. Chen, P. Kim, and R. Song:
%Dirichlet heat kernel estimates for rotationally symmetric L\'{e}vy processes.
%\emph{Proc. Lond. Math. Soc.} (3), 109(1) (2014) 90--120.


\bibitem{CLQ}
W. Chen, C. Li and S. Qi:
A Hopf lemma and regularity for fractional $p$-Laplacians, \emph{AIMS Math.} \textbf{39}, 1477--1495,
2019

\bibitem{Ciomaga}
A. Ciomaga:
On the strong maximum principle for seecond order nonlinear parabolic integro-differential equations,
\emph{Adv. Differ. Equ.} \textbf{17}, 635--671, 2012
%\bibitem{DQT}
%G. D\'{a}vila, A. Quaas and E. Topp:
%Existence, nonexistence and multiplicity results for nonlocal Dirichlet problems,
%\emph{J. Diff. Equations}, to appear, 2019.


\bibitem{DDW}
J. D\'avila, M. del Pino and J. Wei:
Concentrating standing waves for the fractional nonlinear Schr\"odinger equation,
\emph{J. Differential Equations} \textbf{256}, 858--892, 2014

\bibitem{DQT}
G. D\'avila, A. Quass and E. Topp:
Existence, nonexistence and multiplicity results for nonlocal Dirichlet problems,
\emph{J. Differ. Equations} \textbf{266}, 5971--5997, 2019


\bibitem{DLN}
C.-S. Deng, W. Liu, E. Nane:
Finite time blowup of solutions to SPDEs with Bernstein functions of the Laplacian,
\emph{arXiv:2001.00320}, 2020



\bibitem{DKV}
S. Dipierro, A. Karakhanyan and E. Valdinoci:
New trends in free boundary problems, \emph{Adv. Nonlinear Stud.} \textbf{17} 319--332, 2017

\bibitem{DV}
S. Dipierro and E. Valdinoci:
(Non)local and (non)linear free boundary problems, \emph{Discrete Contin. Dyn. Syst. Ser. S} \textbf{11}
465--476, 2018

\bibitem{FJ15}
M. M. Fall and S. Jarohs:
Overdetermined problems with fractional Laplacian.
\emph{ESAIM Control Optim. Calc. Var.} \textbf{21}, 924--938, 2015

\bibitem{FQT}
P. Felmer, A. Quaas and J. Tan:
Positive solutions of the nonlinear Schr\"odinger equation with the fractional Laplacian,
\emph{Proc. Roy. Soc. Edinburgh Sect. A} \textbf{142}, 1237--1262, 2012

\bibitem{FW14}
P. Felmer and W. Wang:
Radial symmetry of positive solutions to equations involving the fractional Laplacian.
\emph{Commun. Contemp. Math.} \textbf{61}, 1350023, 2014

\bibitem{F74}
B. Fristedt:
Sample functions of stochastic processes with stationary, independent increments,
\emph{Advances in Probability and Related Topics}, Vol. \textbf{3}, 241--396, 1974

%\bibitem{GGK}
%I. Gohberg, S. Goldberg and M. A. Kaashoek:
%\emph{Classes of Linear Operators.}, Vol. I, Operator Theory: Advances and Applications 49, Birkh\"auser
%Verlag, Basel, 1990


\bibitem{GS16}
A. Greco and R. Servadei:
Hopf's lemma and constrained radial symmetry for the fractional Laplacian, \emph{Math. Res. Lett.} \textbf{23},
863--885, 2016

\bibitem{IW}
N. Ikeda and S. Watanabe:
On some relations between the harmonic measure and the L\'evy measure for a certain class of Markov processes
\emph{J. Math. Kyoto Univ.} \textbf{2}, 79--95, 1962

\bibitem{IW20}
A. Ishida and K. Wada:
Threshold between short and long-range potentials for non-local Schr\"odinger operators, \emph{Math. Phys. Anal. Geom.} 
\textbf{23}, 32, 2020, https://doi.org/10.1007/s11040-020-09356-0

\bibitem{JW14}
S. Jarohs and T. Weth:
Asymptotic symmetry for parabolic equations involving the fractional Laplacian,
\emph{Discrete Contin. Dyn. Syst.} \textbf{34}, 2581--2615, 2014

\bibitem{JW16}
S. Jarohs and T. Weth:
Symmetry via antisymmetric maximum principles in nonlocal problems of variable order,
\emph{Ann. Mat. Pura Appl.} (4) 195, 273--291, 2016

\bibitem{KKL}
K. Kaleta, M. Kwa\'snicki and  J. L\H{o}rinczi:
Contractivity and ground state domination properties for non-local Schr\"odinger operators, \emph{J. Spectr. Theory}
\textbf{8}, 165--189, 2018



%\bibitem{JW16}
%S. Jarohs and T. Weth:
%Symmetry via antisymmetric maximum principles in nonlocal problems of variable order.
%\emph{Ann. Mat. Pura Appl.} \textbf{195}(4) (2016), 273--291

\bibitem{KL15}
K. Kaleta and J. L\H{o}rinczi:
Pointwise estimates of the eigenfunctions and intrinsic ultracontractivity-type properties of Feynman-Kac
semigroups for a class of L\'evy processes, \emph{Ann. Probab.} \textbf{43}, 1350--1398, 2015

%\bibitem{KL16}
%K. Kaleta and J. L\H{o}rinczi:
%Transition in the decay rates of stationary distributions of L\'evy motion in an energy landscape,
%\emph{Phys. Rev.} \textbf{E93}, 022135, 2016

\bibitem{KL17}
K. Kaleta and J. L\H{o}rinczi:
Fall-off of eigenfunctions of non-local Schr\"odinger operators with decaying potentials, \emph{Potential Anal.}
\textbf{46}, 647--688, 2017

\bibitem{KL19}
K. Kaleta and J. L\H{o}rinczi:
Zero-energy bound state decay for non-lcoal Schr\"o\-din\-ger operators,   \emph{Commun. Math. Phys.} \textbf{374},
2151--2191, 2020

\bibitem{KKLL}
M. Kim, P. Kim, J. Lee and K.A. Lee:
Boundary regularity for nonlocal operators with kernels of variable orders, \emph{J. Funct. Anal.} \textbf{277},
279--332, 2019

\bibitem{KSV}
P. Kim, R. Song, and Z. Vondra\v{c}ek:
Potential theory of subordinate Brownian motions revisited, in \emph{Stochastic Analysis and Applications to
Finance}, Interdiscip. Math. Sci. \textbf{13}, World Scientific, 2012

\bibitem{KMR}
M. Kwa\'{s}nicki, J. Ma\l ecki and M. Ryznar:
Suprema of L\'{e}vy processes, \emph{Ann. Probab.} \textbf{41}, 2047--2065, 2013

\bibitem{KM}
M. Kwa\'snicki and J. Mucha:
Extension technique for complete Bernstein functions of the Laplace operator, \emph{J. Evol. Equ.} \textbf{18},
1341--1379, 2018


\bibitem{LS}
E.H. Lieb and R. Seiringer:
\emph{The Stability of Matter in Quantum Mechanics}, Cambridge University Press, 2010

\bibitem{DQ17}
L.M. del Pezzo and A. Quaas:
A Hopf's lemma and a strong minimum principle for the fractional $p$-Laplacian,
\emph{J. Differential Equations} \textbf{263}, 765-778, 2017

\bibitem{PW}
M.H. Protter and H.F. Weinberger:
\emph{Maximum Principles in Differential Equations}, Springer, 1984

\bibitem{PS}
P. Pucci and J. Serrin:
The strong maximum principle revisited, \emph{J. Differential Equations} \textbf{196}, 1--66, 2004

\bibitem{RO}
X. Ros-Oton:
Nonlocal equations in bounded domains: a survey, \emph{Publ. Mat.} \textbf{60}, 3--26, 2016

\bibitem{ROS}
X. Ros-Oton and J. Serra:
The extremal solution for the fractional Laplacian, \emph{Calc. Var. Partial Differ. Eq.}
\textbf{50}, 723--750, 2014


%\bibitem{QS}
%A. Quaas and B. Sirakov:
%Principal eigenvalues and the Dirichlet problem for fully nonlinear elliptic operators,
%\emph{Adv. Math.} \textbf{218} (1), 105--135, 2008.

\bibitem{SSV}
R. Schilling, R. Song and Z. Vondra\v{c}ek:
\emph{Bernstein Functions}, Walter de Gruyter, 2010

\bibitem{S13}
S. Secchi:
Ground state solutions for nonlinear fractional Schr\"odinger equations in ${\mathbb R}^N$, \emph{J. Math. Phys.}
\textbf{54}, 031501, 2013

\bibitem{S71}
J. Serrin:
A symmetry problem in potential theory, \emph{Arch. Rational Mech. Anal.} \textbf{43}, 304--318, 1971

\bibitem{SV19}
N. Soave and E. Valdinoci:
Overdetermined  problems  for  the  fractional  Laplacian  in  exterior and  annular  sets,
\emph{J. Anal. Math.} \textbf{13}, 101--134, 2019

%\bibitem{W07}
%T. Watanabe:
%Asymptotic estimates of multi-dimensional stable densities and their applications, \emph{Trans. Amer. Math. Soc.}
%\textbf{359} (2007), 2851--2879,

\bibitem{W71}
H. F. Weinberger:
Remark on the preceding paper of Serrin, \emph{Arch. Rational Mech. Anal.} \textbf{43}, 319--320, 1971

\end{thebibliography}
\end{document}